\documentclass[12pt,letterpaper]{amsart}
\usepackage[letterpaper,hmargin=1.25in,vmargin=1.5in]{geometry}
\usepackage{latexsym, amssymb, amsmath}
\usepackage{color}

\newtheorem{theorem}{Theorem}[section]
\newtheorem{lemma}[theorem]{Lemma}
\newtheorem{corollary}[theorem]{Corollary}

\newtheorem{claim}[theorem]{Claim}
\newtheorem{remark}[theorem]{Remark}
\theoremstyle{definition}

\makeatletter
  
  \@addtoreset{equation}{section}
\makeatother

\begin{document}

\title[Complete gradient Yamabe solitons]
{Complete gradient Yamabe solitons
}

\author{Shun Maeta}
\address{Department of Mathematics, Chiba University, 1-33, Yayoicho, Inage, Chiba, 263-8522, Japan.}
\curraddr{}
\email{shun.maeta@faculty.gs.chiba-u.jp}
%\thanks{The author is partially supported by the Grant-in-Aid for Scientific Research (C), No.23K03107, Japan Society for the Promotion of Science.}
%\thanks{}
\subjclass[2020]{53C21, 53C25, 53C20}
\date{}
\dedicatory{}
\keywords{steady gradient Yamabe solitons; shrinking gradient Yamabe solitons; expanding gradient Yamabe solitons; Perelman's conjecture; Yamabe flow}
\commby{}

\begin{abstract}
In this paper, all complete gradient Yamabe solitons of dimension at least three are classified. More precisely, we show that complete steady gradient Yamabe solitons are either rotationally symmetric manifolds or scalar-flat manifolds.
We also classify nontrivial complete shrinking and expanding gradient Yamabe solitons.  
\end{abstract}

\maketitle

% 参考文献スタイルを amsalpha に変更 これでReferences に [ ] がつく
\bibliographystyle{amsalpha}

%%%%%%%%%%%%%%%%%%%%%%%%%%%%%%%%%%%
%%%%%%%%%%%%%%%%%%%%%%%%%%%%%%%%%%%
%%%%%%%%%%%%%%%%%%%%%%%%%%%%%%%%%%%

\section{Introduction}\label{intro}

Ricci flow and Yamabe flow are powerful tools to understand Riemannian manifolds. To analyze these geometric flows, it is important to study their singularities. Ricci solitons and Yamabe solitons are self-similar solutions to the Ricci flow and the Yamabe flow, respectively, and are expected to serve as models of these singularities (cf. \cite{CLN06}). 

The Yamabe flow was introduced by Hamilton \cite{Hamilton89}. 
On closed manifolds, Hamilton proved the long-time existence of the normalized flow. Chow \cite{Chow92} proved its convergence under locally conformally flat metrics with positive Ricci curvature, and Ye \cite{Ye94} removed the curvature assumption. 
Under an initial energy restriction, Schwetlick and Struwe obtained convergence results \cite{SS03}. Brendle proved convergence for arbitrary initial energy in dimensions 3 to 5 \cite{Brendle05}, and in dimension $n\geq6$ under some assumptions on the Weyl tensor \cite{Brendle07}. 

The Yamabe flow on complete noncompact manifolds has also been studied.
First, Ma and An studied existence of the Yamabe flow on locally conformally flat manifolds under some curvature assumptions \cite{MA99}.
Daskalopoulos and Sesum studied the extinction behavior of solutions to the fast diffusion equation, and applied their results to the Yamabe flow \cite{DS08}.
Choi and Daskalopoulos studied the convergence of conformally flat solutions to the Yamabe flow to steady Yamabe solitons, and proved the existence of Type II singularities \cite{CD18}. 
Extinction profiles for solutions to the Yamabe flow were studied by Daskalopoulos, King, and Sesum \cite{DKS19}. 
Schulz developed the existence theory of the Yamabe flow \cite{Schulz19,Schulz20a,Schulz20b}.
Choi, Daskalopoulos, and King carried out detailed asymptotic analysis of Type II singularities for the Yamabe flow and determined their blow-up rates and profiles \cite{CDK21}. Global existence of the Yamabe flow on asymptotically flat manifolds was proved by Ma \cite{Ma21}. Chen and Wang proved convergence to a scalar-flat metric on asymptotically flat manifolds with positive Yamabe constant \cite{CW23}. When the Yamabe constant is nonpositive, Carron, Chen, and Wang studied the rescaled asymptotic behavior of the Yamabe flow on asymptotically flat manifolds \cite{CCW23}.

For the study of Ricci solitons, one of the famous conjectures is Perelman's conjecture, which states that ``any three-dimensional complete noncompact $\kappa$-noncollapsed gradient steady Ricci soliton with positive curvature is rotationally symmetric'' (see \cite{Perelman02}). This conjecture was proved by Brendle \cite{Brendle13}. 
Here, an $n$-dimensional Riemannian manifold $(M^n,g)$ is called {\it rotationally symmetric} if it admits an isometric cohomogeneity-one $\mathrm{SO}(n)$-action and, away from the fixed points of this action, the metric is of the form
$
g=dr^2+a(r)^2g_{\mathrm{unit}}
$
on $I\times\mathbb{S}^{n-1}$, where $I\subset\mathbb{R}$ is an interval, $g_{\mathrm{unit}}$ is the standard round metric on the unit $(n-1)$-dimensional sphere, and $a$ is positive on the regular part. At a fixed point $r_0$, we have $a(r_0)=0$. Smoothness requires $a'(r_0)=\pm1$ and a smooth odd extension of $a$ across $r_0$ (see \cite{Bryant05}).
In the warped product descriptions with warping function $|\nabla F|$, $\bar g_S$ denotes a suitably scaled round metric.
If $r$ is the distance from a critical point of $F$ and $c=F''(0)\neq0$, then $\bar g_S=c^{-2}g_{\mathrm{unit}}$.

In higher dimensions, Brendle also provided interesting results under an asymptotically cylindrical assumption (cf. \cite{Brendle14}).
However, complete resolutions are still far away.

In this paper, we consider a problem analogous to Perelman's conjecture.
First, it was shown that any compact gradient Yamabe soliton is trivial (see \cite{Hsu12}). Therefore complete noncompact gradient Yamabe solitons were studied. This problem was first considered by Daskalopoulos and Sesum. They showed that any locally conformally flat complete gradient Yamabe soliton with positive sectional curvature is rotationally symmetric (cf. \cite{DS13}).
Cao, Sun, and Zhang relaxed the curvature assumption. 
They showed that the universal cover of a nontrivial locally conformally flat complete gradient Yamabe soliton with positive scalar curvature is rotationally symmetric (cf. \cite{CSZ12}). Catino, Mantegazza, and Mazzieri \cite{CMM12} studied a broader class of conformal gradient solitons and obtained corresponding results without assuming local conformal flatness, under the assumption that the Ricci curvature is nonnegative and positive at some point.
Hsu \cite{Hsu14} obtained precise asymptotic results for the conformal factors and curvatures of rotationally symmetric gradient Yamabe solitons under suitable parameter assumptions.
Choi, Daskalopoulos, and King proved the existence of Type II singularities of complete noncompact Yamabe flows whose blow-up limits are rotationally symmetric steady gradient Yamabe solitons \cite{CDK21}.

He \cite[Theorem 1.1 and Corollary 1.2]{He11} claimed a classification result without assuming nonnegative scalar curvature.
The proof relies on Remark 1.5 of \cite{He11}, which asserts that steady gradient Yamabe solitons have nonnegative scalar curvature, citing Daskalopoulos and Sesum.
However, the nonnegativity argument in \cite[Section 3]{DS13} concerns complete rotationally symmetric metrics globally conformal to the standard metric on Euclidean space, with a smooth positive radial conformal factor.
Their argument does not prove the nonnegativity of the scalar curvature in the general warped product setting considered in \cite{He11}.
Since He uses this nonnegativity to reduce the proof to the case where the fiber has positive scalar curvature, the argument as presented in \cite{He11} requires at least the additional assumption of nonnegative scalar curvature.
Under nonnegative scalar curvature and positive scalar curvature at some point, the author also provided results on rotational symmetry for steady gradient Yamabe solitons \cite{Maeta28}. Furthermore, nontrivial nonflat complete gradient Yamabe solitons were classified in dimension two \cite{BM15} and in dimension three \cite{Maeta29}. In this paper, we consider higher dimensions. We assume that the dimension is greater than two.

In this paper, we remove all curvature assumptions and provide a classification of complete steady gradient Yamabe solitons.

\begin{theorem}\label{main}
Let $(M^n,g,F)$ $(n\geq3)$ be a complete, steady gradient Yamabe soliton. 
Then $(M,g)$ is isometric to either a manifold with zero scalar curvature or  a rotationally symmetric manifold
\[
 ([0,+\infty),dr^2)
 \times_{|\nabla F|}
 (\mathbb{S}^{n-1},\bar g_S),
\]
where $\bar g_S$ is a round metric on $\mathbb{S}^{n-1}$, with its scale chosen as above.
In the latter case, the potential function $F$ has a unique critical point.\end{theorem}

As a corollary, we obtain rotational symmetry of complete steady gradient Yamabe solitons that are not scalar-flat.

\begin{corollary}
Any complete steady gradient Yamabe soliton that is not scalar-flat is rotationally symmetric.
\end{corollary}

\begin{remark}
\begin{enumerate}
\item
An elementary argument using the Yamabe soliton equation and \eqref{RT3} below shows that the rotationally symmetric manifold $(M^n,g)$ in Theorem \ref{main} is globally conformally equivalent to the Euclidean space $\mathbb{E}^n$.
\item
There are many nontrivial scalar-flat complete steady gradient Yamabe solitons.
In fact, for any constants $\alpha>0$ and $\beta\in\mathbb{R}$ and any complete Riemannian manifold $(N^{n-1}(0),g_N)$ with a scalar-flat metric, 
\[
(\mathbb{R}\times N^{n-1}(0),dr^2+\alpha^2g_N,F=\alpha r+\beta,\lambda=0)
\]
is a nontrivial scalar-flat complete steady gradient Yamabe soliton. Taking $N^{n-1}(0)$ to be a flat torus gives a non-rotationally symmetric example.
\end{enumerate}
\end{remark}

For shrinking solitons, He considered complete shrinking gradient Yamabe solitons of the form $\mathbb{R}\times_\varphi N^{n-1}$ $(n\geq3)$ and showed that there is a unique complete gradient Yamabe soliton that is not isometric to the product manifold under a suitable assumption on the scalar curvature $\bar R$ of $N$ (cf. Theorem 1.3 in \cite{He11}).
Recently the author showed that under $R\geq\lambda$ and $R>\lambda$ at some point, complete shrinking gradient Yamabe solitons are rotationally symmetric (cf. \cite{Maeta28}). Furthermore, the assumption on the scalar curvature is optimal.
Complete shrinking gradient Yamabe solitons are classified in dimension two by Bernstein and Mettler \cite{BM15} and in dimension three by the author \cite{Maeta29}.

In this paper, we classify complete shrinking gradient Yamabe solitons without any curvature assumption.

\begin{theorem}\label{main2}
Let $(M^n,g,F,\lambda)$ $(n\geq3)$ be a nontrivial complete shrinking gradient Yamabe soliton.
Exactly one of the following holds: 
\begin{enumerate}
\item[$(1)$]
The potential function $F$ has a unique critical point, and $(M,g)$ is isometric to
\[
 ([0,+\infty),dr^2)
 \times_{|\nabla F|}
 (\mathbb{S}^{n-1},\bar g_S),
\]
where $\bar g_S$ is a round metric on $\mathbb{S}^{n-1}$, with its scale chosen as above.

\item[$(2)$]
The potential function $F$ has no critical point, and $(M,g)$ is isometric to
\[
 (\mathbb{R},dr^2)
 \times_{|\nabla F|}
 (N^{n-1},\bar g),
\]
where $\rho(r)=F'(r)>0$ and $(N^{n-1},\bar g)$ is a complete Riemannian manifold with constant scalar curvature $\bar R$.
\end{enumerate}
Moreover, in Case $(2)$, 
\begin{enumerate}
\item[$(a)$]
When $\bar R\leq0$, no such complete solitons exist.
\item[$(b)$]
When $0<\bar R\leq (n-1)(n-2)\lambda^2$, $\rho\equiv \sqrt{\frac{\bar R}{\lambda}}$ and $M$ is isometric to a product manifold.
\item[$(c)$]
When $\bar R>(n-1)(n-2)\lambda^2$, $M$ is isometric to a product manifold with $g=dr^2+\frac{\bar R}{\lambda}\bar g$ or, for fixed $(N,\bar g)$ and $\lambda$, a unique complete gradient Yamabe soliton that is not isometric to the product manifold.
Furthermore, the nonconstant solution satisfies
 \[
 \rho(r)\to+\infty
,\quad \rho'(r)\to-\lambda\qquad (r\to-\infty),
 \]
  \[
 \rho(r)\to\sqrt{\frac{\bar R}{\lambda}}
,\quad \rho'(r)\to 0\qquad (r\to+\infty),
 \]
 and 
 \[
 R>0,\qquad
 \lim_{r\to-\infty}R(r)=0,\quad\text{and}\quad\lim_{r\to+\infty}R(r)=\lambda.
 \]
\end{enumerate}
In $(c)$, uniqueness is understood up to translation of $r$ and addition of a constant to $F$.
\end{theorem}

\begin{remark}
In Case (2), a nonconstant positive warping function on $\mathbb{R}$ exists if and only if
\[
\bar R>(n-1)(n-2)\lambda^2.
\]
\end{remark}

In low dimensions, expanding gradient Yamabe solitons exhibit a greater variety of examples than shrinking and steady solitons (cf. \cite{BM15}, \cite{Maeta29}).
This suggested that classifying expanding gradient Yamabe solitons in higher dimensions would involve many different cases. Nevertheless, in this paper we obtain a complete classification of expanding gradient Yamabe solitons.

\begin{theorem}\label{main3}
Let $(M^n,g,F,\lambda)$ $(n\geq3)$ be a nontrivial complete expanding gradient Yamabe soliton.
Exactly one of the following holds: 
\begin{enumerate}
\item[$(1)$]
The potential function $F$ has a unique critical point, and $(M,g)$ is isometric to
\[
 ([0,+\infty),dr^2)
 \times_{|\nabla F|}
 (\mathbb{S}^{n-1},\bar g_S),
\]
where $\bar g_S$ is a round metric on $\mathbb{S}^{n-1}$, with its scale chosen as above.

\item[$(2)$]
The potential function $F$ has no critical point, and $(M,g)$ is isometric to
\[
 (\mathbb{R},dr^2)
 \times_{|\nabla F|}
 (N^{n-1},\bar g),
\]
where $\rho(r)=F'(r)>0$ and $(N^{n-1},\bar g)$ is a complete Riemannian manifold with constant scalar curvature $\bar R$.
\end{enumerate}
Moreover, in Case $(2)$, 
\begin{enumerate}
\item[$(a)$]
When $\bar R>0$, no such complete solitons exist.
\item[$(b)$]
When $\bar R=0$, for fixed $(N,\bar g)$ and $\lambda$, there exists a unique complete gradient Yamabe soliton that is not isometric to the product manifold.
Furthermore, the nonconstant solution satisfies
 \[
 \rho(r)\to\sqrt{\frac{\bar R}{\lambda}}
,\quad \rho'(r)\to 0 \qquad (r\to-\infty),
 \]
  \[
 \rho(r)\to+\infty
,\quad \rho'(r)\to -\lambda \qquad (r\to+\infty),
 \]
 and 
 \[
 \lambda<R<0,\qquad
 \lim_{r\to-\infty}R(r)=\lambda,\quad\text{and}\quad\lim_{r\to+\infty}R(r)=0.
 \]
\item[$(c)$]
When $\bar R<0$, $M$ is isometric to a product manifold with $g=dr^2+\frac{\bar R}{\lambda}\bar g$ or, for fixed $(N,\bar g)$ and $\lambda$, a unique complete gradient Yamabe soliton that is not isometric to the product manifold.
Furthermore, the nonconstant solution satisfies
 \[
 \rho(r)\to\sqrt{\frac{\bar R}{\lambda}}
,\quad \rho'(r)\to 0 \qquad (r\to-\infty),
 \]
  \[
 \rho(r)\to+\infty
,\quad \rho'(r)\to -\lambda \qquad (r\to+\infty),
 \]
 and 
 \[
 \lambda<R<0,\qquad
 \lim_{r\to-\infty}R(r)=\lambda,\quad\text{and}\quad\lim_{r\to+\infty}R(r)=0.
 \]
\end{enumerate}
In $(b)$ and $(c)$,  uniqueness is understood up to translation of $r$ and addition of a constant to $F$.
\end{theorem}

The remaining sections are organized as follows. 
Section \ref{sectionP} contains some necessary definitions and preliminary geometric results.
 In Section \ref{theorem1}, we prove Theorem \ref{main}.
Section \ref{theorem2} is devoted to the proof of Theorem \ref{main2}.
In Section \ref{theorem3}, the proof of Theorem \ref{main3} is provided.

%%%%%%%%%%%%%%%%%%%%%%%%%%%
%%%%%%%%%%%%%%%%%%%%%%%%%%%
%%%%%%%%%%%%%%%%%%%%%%%%%%%

\section{Preliminaries and lemmas}\label{sectionP}

An $n$-dimensional Riemannian manifold $(M^n,g)$ is called a gradient Yamabe soliton if there exist a smooth function $F$ on $M$ and a constant $\lambda\in \mathbb{R}$ such that 
$\nabla \nabla F=(R-\lambda)g,$
where $R$ is the scalar curvature of $M$.
We write $(M^n,g,F,\lambda)$ for such a soliton.
If $F$ is constant, then $M$ is called trivial.
If $\lambda>0$, $\lambda=0$, or $\lambda<0$, then the Yamabe soliton is called shrinking, steady, or expanding, respectively.
The Riemannian curvature tensor is defined by 
\[
R(X,Y)Z=-\nabla_X\nabla_YZ+\nabla_Y\nabla_XZ+\nabla_{[X,Y]}Z,
\]
for $X,Y,Z\in \mathfrak{X}(M)$, where $\nabla$ is the Levi-Civita connection of $M$. The Ricci tensor $R_{ij}$ is defined by 
$R_{ij}=g^{pq}R_{ipjq},$ where $R_{ijk\ell}=g(R({\partial_i,\partial_j})\partial_k,\partial_\ell).$

Cao, Sun, and Zhang (Theorem 1.2 in \cite{CSZ12}) established the following structure theorem for Yamabe solitons (see also Theorem 1 in \cite{Tashiro65} and Theorem 1.1 in \cite{CMM12}).

\begin{theorem}[\cite{Tashiro65}, \cite{CSZ12}, \cite{CMM12}]\label{CSZ}
Let $(M^n,g,F,\lambda)$ be a nontrivial complete gradient Yamabe soliton. Then $|\nabla F|^2$ is constant on every regular level hypersurface of $F$, and exactly one of the following holds:
\begin{enumerate}
\item[$(1)$]
The potential function $F$ has a unique critical point, and $(M,g)$ is isometric to
\[
 ([0,+\infty),dr^2)
 \times_{|\nabla F|}
 (\mathbb{S}^{n-1},\bar g_S),
\]
where $\bar g_S$ is a round metric on $\mathbb{S}^{n-1}$, with its scale chosen as above.

\item[$(2)$]
The potential function $F$ has no critical point, and $(M,g)$ is isometric to
\[
 (\mathbb{R},dr^2)
 \times_{|\nabla F|}
 (N^{n-1},\bar g),
\]
where $(N^{n-1},\bar g)$ has constant scalar curvature $\bar R$.
\end{enumerate}
Moreover, in Case $(2)$, if $R\geq0$ on $M$, then either $\bar R>0$, or $R=\bar R=0$ and $|\nabla F|\equiv a$ for some constant $a>0$.
In the latter case, $(M,g)$ is isometric to the Riemannian product
\[
(\mathbb{R},dr^2)\times(N^{n-1},a^2\bar g).
\]
\end{theorem}

The following lemma will be used repeatedly below.
\begin{lemma}[Lemma 2.3 in \cite{Maeta28}]\label{zero}
Let $a_1, a_2, a_3, a_4$ be constants. Suppose that $\rho: \mathbb{R} \to \mathbb{R}$ is a smooth positive function satisfying the differential equation
\begin{equation} \label{eq:rho_ode}
\rho\rho'' = a_1- a_2\rho^2 - a_3^2\rho'\rho^2 - a_4^2\rho'^2.
\end{equation}
If $\rho$ satisfies either $\rho' \geq 0$ or $\rho' \leq 0$ on $\mathbb{R}$, and there exists a point $x_0 \in \mathbb{R}$ such that $\rho'(x_0) = 0$, then $\rho' \equiv 0$ on $\mathbb{R}$.
\end{lemma}

%%%%%%%%%%%%%%%%%%%%%
%%%%%%%%%%%%%%%%%%%%%
%%%%%%%%%%%%%%%%%%%%%

\section{Proof of Theorem \ref{main}}\label{theorem1}

In this section we prove Theorem \ref{main}.

\begin{proof}[Proof of Theorem \ref{main}]

If $F$ is constant, then the soliton equation $\nabla \nabla F=Rg$ gives $R\equiv0$. Hence we may assume that $F$ is nonconstant.

We only need to consider Case (2) of Theorem \ref{CSZ}.
In this case, for $a,b,c,d=2,3,\cdots,n,$ the curvature tensors are obtained as follows:
\begin{align*}
R_{1a1b}&=-F'F'''{\bar g}_{ab},\quad R_{1abc}=0,\\
R_{abcd}&=(F')^2{\bar R}_{abcd}+(F'F'')^2(\bar g_{ad}\bar g_{bc}-\bar g_{ac}\bar g_{bd}),\notag
\end{align*}
\begin{align*}
R_{11}&=-(n-1)\frac{F'''}{F'},\quad 
R_{1a}=0,\\
R_{ab}&=\bar R_{ab}-((n-2)(F'')^2+F'F''')\bar g_{ab},\notag
\end{align*}
\begin{align}\label{RT3}
R=(F')^{-2}\bar R-(n-1)(n-2)\Big(\frac{F''}{F'}\Big)^2-2(n-1)\frac{F'''}{F'},
\end{align}
where the curvature tensors with a bar are the curvature tensors of $(N,\bar g)$.
By \eqref{RT3} and the soliton equation $R=F''$, one has the following ODE.
\begin{equation}\label{key}
\rho'\rho^2+(n-1)(n-2)\rho'^2+2(n-1)\rho\rho''=\bar R,
\end{equation}
where one can choose the direction of $r$ such that $\rho(r)=F'(r)>0$.
\\

\noindent
Case 1. $\bar R<0$. 

Assume that there exists $r_0\in\mathbb{R}$ such that $\rho'(r_0)=0$.
By \eqref{key}, we have 
\[
\rho''(r_0)=\frac{\bar R}{2(n-1)\rho(r_0)}<0.
\]
Thus, there is at most one zero of $\rho'$. 
Hence, 
\[
\rho'>0\qquad\text{on}\quad (-\infty,r_0).
\]
Set $s=-r$. We use the same notation $\rho$ for the reparametrized function and $\dot\rho(s)=\frac{d}{ds}\rho(s)$. 
Equation \eqref{key} is
\begin{equation}\label{keys}
-\dot\rho\rho^2+(n-1)(n-2)\dot\rho^2+2(n-1)\rho\ddot\rho=\bar R.
\end{equation}
Combining this, $\dot\rho(s)<0$, and $\rho(s)>0$ on $(-r_0,+\infty)$, we have $\ddot\rho(s)<0$ on $(-r_0,+\infty)$. Since $\rho(s)>0$ on $(-r_0,+\infty)$, we have a contradiction.
Thus, there exists no zero of $\rho'$. 
\\

\noindent
Case 1-a. $\bar R<0$ and $\rho'>0$ on $\mathbb{R}$.

By \eqref{key}, we have $\rho''<0$ on $\mathbb{R}$, which yields a contradiction.
\\

\noindent
Case 1-b. $\bar R<0$ and $\rho'<0$ on $\mathbb{R}$.

Fix $\tilde r\in\mathbb{R}$. Set $a=\rho(\tilde r)$.  
Since $\rho'<0$, $\rho(r)<a$ on $(\tilde r,+\infty)$.

Assume that $\frac{\bar R}{2a^2}\leq\rho'<0$. Then, by \eqref{key}, we have
\begin{align*}
\rho''
&=\frac{1}{2(n-1)\rho}
\Big\{
\bar R-\rho'\rho^2-(n-1)(n-2)\rho'^2
\Big\}\\
&<\frac{1}{2(n-1)\rho}
\Big\{
\bar R-\rho'a^2
\Big\}\\
&\leq\frac{\bar R}{4(n-1)\rho}<0.
\end{align*}
Hence, if $\frac{\bar R}{2a^2}\leq\rho'<0$, then $\rho'$ is strictly decreasing.  
Therefore, we obtain
\[
\rho'
\leq\max
\Big\{
\rho'(\tilde r),\frac{\bar R}{2a^2}\Big\}<0\qquad \text{on}\quad (\tilde r,+\infty),
\]
which is a contradiction, because $\rho$ is a positive function.

Hence, $\bar R<0$ cannot occur.
\\

\noindent
Case 2. $\bar R>0$. 

We show that $\rho'(r)\geq0$ on $\mathbb{R}$.

Assume that there exists $\tilde r\in\mathbb{R}$ such that 
\[
\rho'(\tilde r)<0.
\]
Hence,
\[
\rho'<0\qquad\text{on} \quad (-\infty,r_2)\ni\tilde r,
\]
for some $r_2$. In fact, if there exists a point $r_0\in \mathbb{R}$ such that $\rho'(r_0)=0$, then by \eqref{key}, we have
\[
\rho''(r_0)=\frac{\bar R}{2(n-1)\rho(r_0)}>0.
\]

Set $s=-r$ and $s_2=-r_2$. We show the following.
\begin{claim}
The positive function $\rho$ satisfies $\rho(r)\rightarrow +\infty$ as $r\rightarrow-\infty$.
\end{claim}
\begin{proof}

Assume that $\rho(s)\to c<+\infty$ as $s\to+\infty$.

If $\ddot\rho\geq0$ on $(s_2,+\infty)$, then $\dot\rho$ is nondecreasing.
Hence $\rho$ cannot be bounded from above, which is a contradiction.

Assume that $\ddot\rho\leq0$ on $(s_2,+\infty)$. 
If $\dot\rho\to \alpha>0$ as $s\to+\infty$, then this contradicts $\rho\leq c$. Hence $\dot\rho\to 0$ as $s\to+\infty$.
By \eqref{keys}, 
$\ddot\rho\to \frac{\bar R}{2(n-1)c}>0$ as $s\to+\infty$. 
However, this contradicts the assumption $\ddot\rho\leq0$.

Hence there exists a point $s_0\in(s_2,+\infty)$ such that $\ddot\rho(s_0)=0$.
By differentiating \eqref{keys}, we obtain
\[
\ddot\rho\rho^2+2\rho\dot\rho^2-2(n-1)^2\dot\rho\ddot\rho-2(n-1)\rho\dddot\rho=0.
\]
At $s_0$, 
\[
\dddot\rho=\frac{1}{n-1}\dot\rho^2>0.
\] 
Therefore, there exists exactly one zero of $\ddot\rho$, and 
\begin{equation*}
\left\{
\begin{aligned}
&\ddot\rho<0\qquad\text{on}\quad (s_2,s_0),\\
&\ddot\rho>0\qquad\text{on}\quad (s_0,+\infty).
\end{aligned}
\right.
\end{equation*}
Hence $\rho$ cannot be bounded from above, which is a contradiction.
\end{proof}
Set $G(\rho)=\dot\rho$.
\begin{align*}
\frac{d}{d\rho}
\Big(
\rho^{\frac{n-2}{2}}G(\rho)
\Big)
&=
\frac{1}{2(n-1)G(\rho)}\bar R \rho^{\frac{n-4}{2}}+\frac{1}{2(n-1)}\rho^{\frac{n}{2}}\\
&>\frac{1}{2(n-1)}\rho^{\frac{n}{2}}.
\end{align*}
By integrating both sides of the inequality, we have
\[
G(\rho)>\frac{1}{(n-1)(n+2)}\rho^2+C\rho^{-\frac{n-2}{2}},
\]
for some constant $C$.
For sufficiently large $\rho$, we obtain
\[
G(\rho)>\frac{1}{2(n-1)(n+2)}\rho^2.
\]
Hence, 
\[
\frac{d}{ds}
\Big(
-\rho^{-1}
\Big)
>\frac{1}{2(n-1)(n+2)}.
\]
Take $s_\ell$ sufficiently large. For $s>s_\ell$, integrating both sides of the above inequality, 
we have
\[
-\rho^{-1}(s)+\rho^{-1}(s_\ell)
>\frac{1}{2(n-1)(n+2)}(s-s_\ell).
\]
Taking $s\to +\infty$, we have a contradiction.

Hence there exists no $\tilde r$ such that $\rho'(\tilde r)<0$.
Therefore $\rho'(r)\geq0$ on $\mathbb{R}$.

If there exists $r_0$ such that $\rho'(r_0)=0$, then by Lemma \ref{zero}, we have $\rho'\equiv0$ on $\mathbb{R}$. By \eqref{key}, we have $\bar R=0$, which is a contradiction.
Therefore, we obtain $\rho'(r)>0$ on $\mathbb{R}$.

Set $s=-r$. Fix $s_1$ and set $\alpha=\rho(s_1)$.
Since $\dot\rho(s)=\frac{d}{ds}\rho(s)<0$, we have
\[
\rho<\alpha\qquad\text{on}\quad(s_1,+\infty).
\]
Take a negative constant $\beta$ such that 
\[
\beta\alpha^2-(n-1)(n-2)\beta^2\geq-\frac{\bar R}{2}.
\]
If $\dot\rho<-\varepsilon$ for some $\varepsilon>0$ on $(s_1,+\infty)$, then we have a contradiction, because $\rho>0$. 
Hence we may assume that there exists a maximal interval $(s_3,s_4)\subset(s_1,+\infty)$ on which $\dot\rho\geq\beta$. In this case, by \eqref{keys},
\begin{align*}
\ddot\rho
&\geq \frac{1}{2(n-1)\rho}
\Big(
\bar R+\beta\alpha^2-(n-1)(n-2)\beta^2
\Big)
\geq\frac{\bar R}{4(n-1)\alpha}>0.
\end{align*}
Hence $s_4$ must be $+\infty$. However, since $\ddot\rho>\frac{\bar R}{4(n-1)\alpha}$ on $(s_3,+\infty)$, we have a contradiction, because $\dot\rho<0$. 
\\

\noindent
Case 3. $\bar R=0$.

We show that $\rho'(r)\geq0$ on $\mathbb{R}$.

Differentiating \eqref{key}, we have
\begin{equation}\label{key3}
\rho''\rho^2+2\rho\rho'^2+2(n-1)^2\rho'\rho''+2(n-1)\rho\rho'''=0.
\end{equation}
Assume that there exists $\tilde r\in\mathbb{R}$ such that 
\[
\rho'(\tilde r)<0,
\]
that is, $\rho'(r)<0$ on a maximal interval $(\tilde r_1,\tilde r_2)\ni\tilde r$.
If $\tilde r_1$ and $\tilde r_2$ are finite, then $\rho'(\tilde r_1)=\rho'(\tilde r_2)=0$.
We note that if there exists $\tilde r_0\in(\tilde r_1,\tilde r_2)$ such that $\rho''(\tilde r_0)=0$, by \eqref{key3}, we obtain
\begin{equation}\label{rho3}
\rho'''(\tilde r_0)=-\frac{1}{n-1}\rho'^2(\tilde r_0)<0.
\end{equation}
Hence if there exists a zero of $\rho''$, then it is unique, and $\rho''$ crosses zero from positive to negative.

Assume that $\tilde r_1$ and $\tilde r_2$ are finite. In this case, $\rho(\tilde r_1)>\rho(\tilde r_2)$.
There exists $\tilde r_0\in(\tilde r_1,\tilde r_2)$ such that $\rho''(\tilde r_0)=0$. 
However, since $\rho$ is smooth, the above argument yields a contradiction.

Assume that $\tilde r_1$ is finite and $\tilde r_2=+\infty$. 
In this case we have
$\rho'(r)<0$ on $(\tilde r_1,+\infty)$. 
If there exists $\tilde r_0\in(\tilde r_1,+\infty)$ such that $\rho''(\tilde r_0)=0$, then  
the above argument shows that $\rho''<0$ on $(\tilde r_0,+\infty)$. 
Combining this, $\rho>0$, and $\rho'<0$ on $(\tilde r_1,+\infty)$, we have a contradiction.
Hence there exists no zero of $\rho''$. Since $\rho'(\tilde r_1)=0$, we have $\rho''<0$ on $(\tilde r_1,+\infty)$. Positivity of $\rho$ yields a contradiction.

Assume that $\tilde r_1=-\infty$ and $\tilde r_2$ is finite.
If there exists $\tilde r_0\in(-\infty,\tilde r_2)$ such that $\rho''(\tilde r_0)=0$, then $\rho''<0$ on $(\tilde r_0,\tilde r_2)$. Since $\rho'(\tilde r_2)=0$, we have a contradiction. Hence there exists no zero of $\rho''$ on $(-\infty,\tilde r_2)$.
In this case, since $\rho'(\tilde r_2)=0$, we have $\rho''>0$ on $(-\infty,\tilde r_2)$. Therefore, $\rho\to+\infty$ as $r\to-\infty$. 

Assume that $\tilde r_1=-\infty$ and $\tilde r_2=+\infty$, that is, $\rho'<0$ on $\mathbb{R}$.
If there exists $\tilde r_0\in\mathbb{R}$ such that $\rho''(\tilde r_0)=0$, then we have $\rho''<0$ on $(\tilde r_0,+\infty)$. Since $\rho>0$, we have a contradiction.

Hence we obtain $\rho''>0$ on $\mathbb{R}$ or $\rho''<0$ on $\mathbb{R}$. 
If $\rho''>0$ on $\mathbb{R}$, then we have $\rho\to +\infty$ as $r\to-\infty$.
If $\rho''<0$ on $\mathbb{R}$, then positivity of $\rho$ yields a contradiction.

Therefore, we obtain $\rho'<0$ on $(-\infty, \tilde r_2)$ and $\rho\to+\infty$ as $r\to-\infty$. Set $s=-r$ and $G(\rho)=\dot\rho$. By the same argument as in Case 2, we obtain a contradiction. 

Therefore, there exists no $\tilde r$ such that $\rho'(\tilde r)<0$, and we obtain $\rho'(r)\geq0$ on $\mathbb{R}$.
%If there exists $\tilde r_0\in\mathbb{R}$ such that $\rho''(\tilde r_0)=0$, we have 
%$\rho''<0$ on $(\tilde r_0,+\infty)$.

If there exists $r_0\in\mathbb{R}$ such that $\rho'(r_0)=0$, then by Lemma \ref{zero}, we have $\rho'\equiv0$ on $\mathbb{R}$. By \eqref{key} and the soliton equation $\nabla\nabla F=Rg$, we have $R=\bar R=0$.

We consider $\rho'(r)>0$ on $\mathbb{R}$.

By \eqref{key},
\begin{align*}
\rho''
&= -\frac{\rho\rho'}{2(n-1)}-\frac{(n-2)\rho'^2}{2\rho}
<0.
\end{align*}
Since $\rho>0$ on $\mathbb{R}$, we have a contradiction.
\end{proof}

%%%%%%%%%%%%%%%%%%%%%%%
%%%%%%%%%%%%%%%%%%%%%%%
%%%%%%%%%%%%%%%%%%%%%%%

\section{Proof of Theorem \ref{main2}}\label{theorem2}

In this section, we classify complete shrinking gradient Yamabe solitons.

\begin{proof}[Proof of Theorem \ref{main2}]
We only need to consider Case (2) of Theorem \ref{CSZ}. We choose the orientation of $r$ so that $\rho=F'(r)>0$. 
By the same argument as in the proof of Theorem \ref{main}, we obtain
\begin{equation}\label{3.2}
\lambda\rho^2+\rho'\rho^2+(n-1)(n-2)\rho'^2+2(n-1)\rho\rho''=\bar R.
\end{equation}
When $\bar R\leq0$, the same argument shows that there exist no positive solutions defined on all of $\mathbb{R}$.

Hence we consider $\bar R>0$.
Set 
\[
\bar R_\lambda=\bar R-(n-1)(n-2)\lambda^2.
\]
By \eqref{3.2},
\[
\rho''=\frac{\bar R-\lambda\rho^2-\rho'\rho^2-(n-1)(n-2)\rho'^2}{2(n-1)\rho}.
\]
If 
\[
\rho'>\sqrt{\frac{\bar R}{(n-1)(n-2)}}
\]
on $\mathbb{R}$, then the positivity of $\rho$ yields a contradiction.
On $\Big\{r\in\mathbb{R}~|~\rho'(r)=\sqrt{\frac{\bar R}{(n-1)(n-2)}}\Big\}$, 
\begin{align*}
\rho''
&=-\frac{1}{2(n-1)}\left(\lambda+\sqrt{\frac{\bar R}{(n-1)(n-2)}}\right)\rho<0.
\end{align*}
Hence we have
\[
\rho'<\sqrt{\frac{\bar R}{(n-1)(n-2)}}
\]
on $\mathbb{R}$.
Differentiating \eqref{3.2}, we have
\begin{equation}\label{3.2.2}
2\lambda\rho\rho'+\rho''\rho^2+2\rho\rho'^2+2(n-1)^2\rho'\rho''+2(n-1)\rho\rho'''=0.
\end{equation}
If there exists $r_0\in\mathbb{R}$ such that $\rho''(r_0)=0$, then 
\begin{equation}\label{p5.1}
\rho'''=-\frac{1}{n-1}\rho'(\lambda+\rho')
\end{equation}
at $r_0$.

In the following we consider the ODE on $\left\{r\in\mathbb{R}~|~\rho'<-\lambda\right\}$. 
Suppose that there exists a point $r_-$ such that $\rho'(r_-)<-\lambda$ and $\rho''(r_-)\leq0$. 
At every zero of $\rho''$ in the region $\rho'<-\lambda$, \eqref{p5.1} gives $\rho'''<0$.
Thus, $\rho''$ cannot cross zero from negative to positive.
Consequently, $\rho'$ remains non-increasing and strictly below $-\lambda$ for $r\geq r_-$, contradicting the positivity of $\rho$.
Therefore, on $\left\{r\in\mathbb{R}~|~\rho'<-\lambda\right\}$, we have $\rho''>0$.
Hence if there exists such an interval, $\left\{r\in\mathbb{R}~|~\rho'<-\lambda\right\}=(-\infty,r_2)$ for some $r_2$.

On $(-\infty,r_2)$, $\rho'<-\lambda$ and $\rho''>0$.
Therefore, for some $r_1<r_2$ and some $\varepsilon>0$, we have
\begin{equation*}
\left\{
\begin{aligned}
&\rho\to+\infty~\text{as}~r\to-\infty\\
&\rho'\leq-\lambda-\varepsilon \quad\text{on}\quad (-\infty,r_1).
\end{aligned}
\right.
\end{equation*}

Set $s=-r$ and $s_1=-r_1$.
Then 
\[\dot\rho\geq\lambda+\varepsilon,\quad \ddot\rho>0 \quad \text{on} \quad (s_1,+\infty).
\]
Set $G(\rho)=\dot\rho$.
By \eqref{3.2}, 
\begin{align*}
\frac{d}{d\rho}G(\rho)
=\frac{\ddot\rho}{\dot\rho}
=\frac{\bar R}{2(n-1)\rho\dot\rho}
-\frac{\lambda\rho}{2(n-1)\dot\rho}
+\frac{\rho}{2(n-1)}
-\frac{n-2}{2\rho}\dot\rho.
\end{align*}
Since $\dot\rho\geq\lambda+\varepsilon$, 
\[
\frac{d}{d\rho}
\Big(
\rho^{\frac{n-2}{2}}G(\rho)
\Big)
=
\frac{1}{2(n-1)}\left(1-\frac{\lambda}{\dot\rho}\right)\rho^\frac{n}{2}
+\frac{\bar R}{2(n-1)\dot\rho}\rho^{\frac{n-4}{2}}.
\]
Integrating both sides of the above equation, we obtain
\[
G(\rho)
\geq
\frac{\varepsilon}{(n-1)(n+2)(\lambda+\varepsilon)}\rho^2+C\rho^{-\frac{n-2}{2}},
\]
for some $C$. For sufficiently large $\rho$, we have
\[
G(\rho)\geq
\frac{\varepsilon}{2(n-1)(n+2)(\lambda+\varepsilon)}\rho^2.
\]
For sufficiently large $s_\ell$, by integrating both sides of the inequality for $s_\ell<s$, we obtain
\[
-\rho^{-1}(s)+\rho^{-1}(s_\ell)\geq \frac{\varepsilon}{2(n-1)(n+2)(\lambda+\varepsilon)}(s-s_\ell).
\]
Taking $s\to+\infty$, we have a contradiction.

Therefore, $\left\{r\in\mathbb{R}~|~\rho'<-\lambda\right\}$ is empty, and $\rho'\geq-\lambda$ on $\mathbb{R}$. If there exists $r_{-\lambda}\in\mathbb{R}$ such that $\rho'(r_{-\lambda})=-\lambda$, then $\rho'$ attains its minimum at $r_{-\lambda}$, so $\rho''(r_{-\lambda})=0$. By \eqref{3.2} again, $\bar R=(n-1)(n-2)\lambda^2$. By uniqueness of solutions to the ODE, we have
\[
\rho(r)=\rho(r_{-\lambda})-\lambda(r-r_{-\lambda}).
\] 
However, the positivity of $\rho$ yields a contradiction.
Therefore,  we have $R>0$.
Therefore, we obtain 
\begin{equation}\label{rhopbounded}
-\lambda<\rho'<\sqrt{\frac{\bar R}{(n-1)(n-2)}}.
\end{equation}
We introduce $\eta$ as follows.
\begin{equation}\label{etadef}
\eta=\rho^{n-2}\rho'^2+\tilde \eta(\rho),
\end{equation}
where we define 
\[
\tilde \eta(u)=\frac{\lambda}{n(n-1)}u^n-\frac{\bar R}{(n-1)(n-2)}u^{n-2}.
\]
Combining
\begin{equation}\label{tetap}
\frac{d}{dr}\tilde \eta(\rho)=\frac{\rho^{n-3}\rho'}{n-1}\Big\{\lambda\rho^2-\bar R\Big\}
\end{equation}
with
\[
\frac{d}{dr}(\rho^{n-2}\rho'^2)
=\frac{\rho^{n-3}\rho'}{n-1}\Big(
\bar R-\lambda\rho^2-\rho'\rho^2
\Big),
\]
we have
\begin{equation}\label{etap}
\eta'=-\frac{\rho^{n-1}\rho'^2}{n-1}\leq0.
\end{equation}
Therefore, $\eta$ is a non-increasing function.
Set $\tilde \eta_{\min}=\tilde \eta\left(\sqrt{\frac{\bar R}{\lambda}}\right)$.  Hence 
\begin{equation}\label{etaboundedbelow}
\eta\geq\tilde \eta_{\min}.
\end{equation}
By the definition of $\tilde \eta$, if $\rho\to+\infty$, then $\tilde \eta(\rho)\to+\infty$.

In what follows, we assume that $\rho$ is nonconstant.
\begin{claim}\label{etar-infty}
$\eta\to + \infty$ as $r\to-\infty$.
\end{claim}
\begin{proof}
Assume that $\eta\to c$ as $r\to-\infty$, for some $c<+\infty$.
Combining this with \eqref{etaboundedbelow}, we have
\[
\tilde \eta_{\min}\leq\eta\leq c.
\]
Since $\tilde \eta(\rho)\leq\eta\leq c$ and $\tilde \eta(u)\to+\infty$ as $u\to+\infty$, $\rho$ is bounded.

Suppose that there exists a sequence $r_j\to-\infty$ such that  
\[
(\rho(r_j),\rho'(r_j))\to(\rho_*,\rho'_*),
\]
where $\rho_*>0$.
Set 
\[
\alpha(r)=(\rho(r),\rho'(r)),
\beta(r)=\eta(\rho(r),\rho'(r))
\]
and 
\[
\alpha_j(r)=\alpha(r_j+r).
\]
At $r=0$, 
\[
\alpha_j(0)=\alpha(r_j)\to(\rho_*,\rho'_*).
\]
For $\rho>0$, the vector field is smooth.
Hence by continuous dependence on initial data, on a sufficiently small common interval $[0,\tau]$, 
\[\alpha_j(r)\to\alpha_*(r)=(\xi_1(r),\xi_2(r))
\]
with $\xi_1(r)>0$.
Integrating \eqref{etap} from $r_j$ to $r_j+\tau$, we obtain
\begin{align*}
\eta(r_j+\tau)-\eta(r_j)
&=-\frac{1}{n-1}\int_{r_j}^{r_j+\tau}\rho^{n-1}(t)\rho'^2(t)dt\\
&=-\frac{1}{n-1}\int_{0}^{\tau}\rho^{n-1}(r_j+t)\rho'^2(r_j+t)dt.
\end{align*} 
Taking $r_j\to-\infty$, we have
\[
0=-\frac{1}{n-1}\int_{0}^{\tau}\xi_1^{n-1}(t)\xi_2^2(t)dt.
\]
Hence $\xi_2(r)\equiv0$ and $\xi_1\equiv \rho_*$ on $[0,\tau]$.
By \eqref{3.2}, $\rho_*=\sqrt{\frac{\bar R}{\lambda}}$. 
Thus,
\[
(\rho_*,\rho'_*)
=
\Big(
\sqrt{\frac{\bar R}{\lambda}},0\Big).
\]
Therefore,
\begin{align*}
\eta(\alpha_*(r))
&=\lim_{j\to+\infty}\eta(\alpha_j(r))\\
&=\lim_{j\to+\infty}\beta(r_j+r)\\
&=c.
\end{align*}
Since 
\[
\alpha_*(r)\equiv \left(\sqrt{\frac{\bar R}{\lambda}},0\right),
\]
we have $c=\tilde\eta_{\min}$.
The inequalities
\[
\tilde\eta_{\min}\leq\eta\leq c
\]
imply that $\eta$ is constant. By \eqref{etap}, $\rho'\equiv0$, contradicting the assumption that $\rho$ is nonconstant. 
This means that if $\rho$ is nonconstant and $c<+\infty$, $\rho(r)\to0$ as $r\to-\infty$. We show that this cannot occur.

\begin{claim}\label{rhor-infty}
It cannot occur that $\rho(r)\to0$ as $r\to-\infty$.
\end{claim}

\begin{proof}
Assume that $\rho(r)\to0$ as $r\to-\infty$. Then for some $r_{-\ell}$, 
\[
\rho(r)<\sqrt{\frac{\bar R}{\lambda}}
\]
on $(-\infty,r_{-\ell})$. 

If there exists $r_0\in (-\infty,r_{-\ell})$ such that $\rho'(r_0)=0$, then by \eqref{3.2}, we have $\rho''(r_0)>0$.
This yields that $\rho'<0$ on $(-\infty,r_0)$.
However, 
\[
\rho(r)=-\int_{r}^{r_0}\rho'(t)dt+\rho(r_0)>\rho(r_0)>0,
\]
which contradicts
\[
\lim_{r\to-\infty}\rho(r)=0.
\]
Therefore, there exists no $r_0$ such that $\rho'(r_0)=0$.
Hence $\rho'>0$ on $(-\infty,r_{-\ell_2})$ for $r_{-\ell_2}<r_{-\ell}$.

If there exists $r_0\in(-\infty,r_{-\ell_2})$ such that $\rho''(r_0)=0$,
 then, by \eqref{3.2.2},  
 \[
 \rho'''(r_0)=-\frac{1}{n-1}\rho'(r_0)(\lambda+\rho'(r_0))<0.
 \]
 Hence $\rho''>0$ on $(-\infty,r_0)$. 
 Since 
 \[
 \int_{-\infty}^{r_0}
\rho'(t)dt=\rho(r_0)<+\infty,
 \]
 $\rho'(r)\to0$ as $r\to-\infty$.
By \eqref{3.2}, $\rho''\to+\infty$ as $r\to-\infty$, which yields $\rho''>1$ on $(-\infty,r_{-\ell_3})$ for $r_{-\ell_3}<r_{-\ell_2}$.
For any $a<r_{-\ell_3}$, we have
\[
\rho'(r_{-\ell_3})-\rho'(a)
=\int_{a}^{r_{-\ell_3}}\rho''(t)dt
>r_{-\ell_3}-a.
\]
Letting $a\to-\infty$ and using $\rho'(a)\to0$, we obtain a contradiction.
Therefore, there exists no $r_0$, and we have $\rho''>0$ or $\rho''<0$ on $(-\infty,r_{-\ell_2})$. If $\rho''>0$, the same argument shows a contradiction.
If $\rho''<0$, then the positivity of $\rho$ yields a contradiction.
Therefore, it cannot occur that $\rho(r)\to0$ as $r\to-\infty$.
This completes the proof of Claim \ref{rhor-infty}.

\end{proof}
By Claim \ref{rhor-infty}, we obtain $\eta(r)\to+\infty$ as $r\to-\infty$.
This completes the proof of Claim \ref{etar-infty}.
\end{proof}

Combining Claim \ref{etar-infty}, \eqref{rhopbounded}, and \eqref{etadef}, we have $\rho(r)\to+\infty$ as $r\to-\infty$.
Hence for some $r_{-\ell}$,
\[
\rho>\sqrt{\frac{\bar R}{\lambda}}
\]
on an interval $(-\infty,r_{-\ell})$.
If there exists $r_0\in(-\infty,r_{-\ell})$ such that $\rho'(r_0)=0$, then by \eqref{3.2}, we have $\rho''(r_0)<0$. Therefore $\rho'>0$ on $(-\infty,r_0)$ and 
\[
\rho(r)=-\int_r^{r_0}\rho'(t)dt+\rho(r_0)<\rho(r_0).
\]
Taking $r\to-\infty$, we have a contradiction and hence there exists no zero of $\rho'$ on $(-\infty,r_{-\ell})$. This yields that 
\[
\rho'>0\quad \text{or} \quad \rho'<0
\]
on $(-\infty,r_{-\ell})$.
If $\rho'>0$, then the same argument yields a contradiction.
Hence $\rho'<0$ on $(-\infty,r_{-\ell})$.

Set $s=-r$ and $s_\ell=-r_{-\ell}$ and
$G(\rho)=\dot\rho$. By \eqref{rhopbounded}, 
\[
0
<G(\rho)
<\lambda
\]
for $\rho>\rho(s_\ell)$.
The same calculation as before shows 
\[
\frac{d}{d\rho}G(\rho)
=\frac{1}{2(n-1)\rho G(\rho)}
\Big\{
\bar R+\rho^2(G(\rho)-\lambda)
-(n-1)(n-2)G^2(\rho)
\Big\}.
\]
For any positive constant $\varepsilon<\lambda$, if $G(\rho)\leq\lambda-\varepsilon$ for sufficiently large $\rho$, then 
\begin{align*}
\frac{d}{d\rho}G(\rho)
&\leq 
\frac{1}{2(n-1)\rho G(\rho)}
\Big\{
\bar R-\varepsilon\rho^2
-(n-1)(n-2)G^2(\rho)
\Big\}\\
&\leq 
\frac{1}{2(n-1) G(\rho)}
\left\{
-\varepsilon\rho
+\frac{\bar R}{\rho}
\right\}\\
&\leq
-\frac{\varepsilon \rho}{4(n-1) G(\rho)}\\
&<-\frac{\varepsilon \rho}{4(n-1) \lambda}.
\end{align*}
Thus, if $G$ reaches or falls below $\lambda-\varepsilon$, it remains below this level for all larger $\rho$. Integrating both sides of the inequality, we have
\begin{align*}
0
&<G(\rho)\\
&<\int^\rho-\frac{\varepsilon t}{4(n-1) \lambda}dt+C\\
&=-\frac{\varepsilon}{8(n-1)\lambda}\rho^2+C.
\end{align*}
Taking $\rho\to+\infty$, we have a contradiction.
Hence, $\dot\rho\to\lambda$ as $s\to+\infty$. 

Therefore,
\[
\rho\to+\infty\quad\text{and}\quad\rho'\to-\lambda
\]
as $r\to-\infty$.
\\

\noindent
Case 1. $\bar R_\lambda\leq0$.

We have
\[
\frac{d}{d\rho}G(\rho)
=\frac{1}{2(n-1)\rho G(\rho)}
\Big\{
\bar R_\lambda
+(\lambda-G(\rho))
\Big(
(n-1)(n-2)(\lambda+G(\rho))-\rho^2
\Big)
\Big\}.
\]
Since $\bar R_\lambda\leq0$ and $G(\rho)<\lambda$, if $\rho$ is large enough, we have
\[
\frac{d}{d\rho}G(\rho)
<0,
\]
which is a contradiction because $G(\rho)<\lambda$ and $G(\rho)\to\lambda$ as $\rho\to+\infty$.

Therefore, there exists no nonconstant solution and 
\[
\rho\equiv \sqrt{\frac{\bar R}{\lambda}}.
\]

\noindent
Case 2. $\bar R_\lambda>0$.

Set 
\[
\theta_\gamma(\rho)=-\lambda+\frac{\gamma}{\rho^2}
\]
and
\[
\Theta(\theta_\gamma)
=2(n-1)\rho\theta_\gamma\frac{d\theta_\gamma}{d\rho}
+(n-1)(n-2)\theta_\gamma^2
+\rho^2(\theta_\gamma+\lambda)
-\bar R.
\]
If $\rho'=\theta_\gamma$, we have
\begin{align}\label{rp-t}
\frac{d}{dr}&(\rho'-\theta_\gamma)\\
&=-\frac{1}{2(n-1)\rho}
\Big\{
-\bar R+\lambda\rho^2+\theta_\gamma\rho^2+(n-1)(n-2)\theta_\gamma^2+2(n-1)\rho\theta_\gamma\frac{d\theta_\gamma}{d\rho}\Big\}\notag\\
&=-\frac{1}{2(n-1)\rho}\Theta(\theta_\gamma).\notag
\end{align}
Since 
\[
\frac{d\theta_\gamma}{d\rho}=-2\gamma\rho^{-3},
\]
we have
\begin{align*}
\Theta(\theta_\gamma)
&=2(n-1)(4-n)\lambda\gamma\rho^{-2}
+(n-1)(n-6)\gamma^2\rho^{-4}\\
&\qquad
+(n-1)(n-2)\lambda^2+\gamma-\bar R.
\end{align*}
Take a sufficiently large $\rho_\ell$. For $\rho\geq\rho_\ell$ we have
\begin{align}\label{Theta1}
\Theta\left(\theta_\frac{\bar R_\lambda}{2}\right)
&=(n-1)(4-n)\lambda\bar R_\lambda\rho^{-2}
+\frac{(n-1)(n-6)}{4}\bar R_\lambda^2\rho^{-4}
-\frac{1}{2}\bar R_\lambda<0
\end{align}
and
\begin{equation}\label{Theta2}
\Theta\left(\theta_{2\bar R_\lambda}\right)
=4(n-1)(4-n)\lambda\bar R_\lambda\rho^{-2}
+4(n-1)(n-6)\bar R_\lambda^2\rho^{-4}
+\bar R_\lambda>0.
\end{equation}
Furthermore, by the definition of $\theta_\gamma$, for $\rho\geq\rho_\ell$,
\begin{equation}\label{p34}
-\lambda<\theta_{\frac{\bar R_\lambda}{2}}<\theta_{2\bar R_\lambda}<-\frac{\lambda}{2}.
\end{equation}
We consider the initial time $\tilde r_0$ such that $\rho(\tilde r_0)>\rho_\ell$
and 
\[
\theta_{\frac{\bar R_\lambda}{2}}(\rho(\tilde r_0))
\leq
\rho'(\tilde r_0)
\leq
\theta_{2\bar R_\lambda}(\rho(\tilde r_0)).
\]
On the lower and upper boundary curves, respectively, 
%Assume that $\tilde r_1>\tilde r_0$ is the first time at which $\rho(\tilde r_1)=\rho_\ell$.
\eqref{rp-t}, \eqref{Theta1}, and \eqref{Theta2} give
\begin{align*}
&\frac{d}{dr} 
\left(
\rho'-\theta_{\frac{\bar R_\lambda}{2}}
\right)
=-\frac{1}{2(n-1)\rho}\Theta\left(\theta_{\frac{\bar R_\lambda}{2}}\right)>0,\\
&\frac{d}{dr} 
\left(
\rho'-\theta_{2{\bar R_\lambda}}
\right)
=-\frac{1}{2(n-1)\rho}\Theta\left(\theta_{2{\bar R_\lambda}}\right)
<0.
\end{align*}
Hence the solution remains between the two boundary curves as long as $\rho\geq\rho_\ell$.
For each $L>\rho_\ell$, we consider a solution 
\[
(u_L(t),v_L(t))
\]
with initial condition 
\[
u_L(0)=L
\]
and
\[
v_L(0)=-\lambda+\frac{\bar R_\lambda}{L^2}.
\]
Here $u_L$ solves \eqref{3.2}, and $v_L=u_L'$.
At a finite time $T_L>0$,
\[
u_L(T_L)=\rho_\ell.
\]
Set $r=t-T_L$, 
\[
x_L(r)=u_L(r+T_L)
\]
and 
\[
y_L(r)=v_L(r+T_L).
\]
Then 
\[
x_L(0)=u_L(T_L)=\rho_\ell.
\]
For $-T_L\leq r\leq0$, the trapping estimates give 
\[
\theta_{\frac{\bar R_\lambda}{2}}(x_L(r))
\leq
y_L(r)
\leq
\theta_{2\bar R_\lambda}(x_L(r)).
\]
Hence 
\begin{equation}\label{p39.1}
L-\rho_\ell
=\int_{-T_L}^0-y_L(r)dr.
\end{equation}
Combining this with \eqref{p34}, we have
\begin{equation}\label{p39.2}
\lambda>-y_L>\frac{\lambda}{2}.
\end{equation}
By \eqref{p39.1} and \eqref{p39.2}, we have
\[
L-\rho_\ell
<\lambda T_L
\] 
and
\[
L-\rho_\ell
>
\frac{\lambda}{2} T_L.
\]
Therefore, we obtain
\[
\frac{L-\rho_\ell}{\lambda}
<T_L
<\frac{2(L-\rho_\ell)}{\lambda}.
\]
Taking $L\to+\infty$, we have $T_L\to\infty$.
Therefore, if we fix the time 0 at which $\rho=\rho_\ell$, the starting time $-T_L$ goes to $-\infty$. 

By passing to a subsequence along which the initial slopes $y_L(0)$ converge and using continuous dependence of ODE solutions on initial data,
we obtain a positive solution defined for all $r\leq0$.
Integrating \eqref{p39.2}, and passing to the limit, we obtain 
\[
-\lambda r+\rho_\ell\geq x_L(r)\geq -\frac{\lambda}{2}r+\rho_\ell.
\]
Thus the limiting solutions cannot blow up in finite backward time.

Since $\bar R_\lambda>0$, if there exists 
$r_{-\lambda}$ such that $\rho'(r_{-\lambda})=-\lambda$, 
\[
\rho''(r_{-\lambda})=\frac{\bar R_\lambda}{2(n-1)\rho}>0
.\]
Hence, $\rho'>-\lambda$ throughout its forward interval of existence. 
Furthermore, since $\eta(r)\leq\eta(0)$ and \eqref{etaboundedbelow} holds, $\rho$ is bounded from above.

\begin{claim}
The function $\rho$ is bounded away from zero on its forward interval of existence.
\end{claim}
\begin{proof}
Set 
\[
0<\delta<\rho_\ell,\quad\text{and}\quad \delta^2\leq\frac{\bar R_\lambda}{2\lambda}.
\]
For 
$\rho\leq\delta$ and 
$-\lambda<\rho'\leq0$,
\begin{align*}
\rho''
&>\frac{1}{2(n-1)\rho}
\left\{
\bar R_\lambda-(\lambda+\rho')\rho^2
\right\}\\
&\geq\frac{1}{2(n-1)\rho}
\left\{
\bar R_\lambda-(\lambda+\rho')\delta^2
\right\}\\
&\geq\frac{\bar R_\lambda}{4(n-1)\rho}
\left(
1-\frac{\rho'}{\lambda}
\right)>0.
\end{align*}

On each interval where $0<\rho\leq\delta$ and $\rho'<0$, we regard $\rho'$ as a function of $\rho$. On each such interval, we have
\begin{align*}
\frac{d\rho'}{d\rho}
&\leq\frac{\bar R_\lambda}{4(n-1)\rho\rho'}\\
&\leq-\frac{\bar R_\lambda}{4(n-1)\lambda\rho}.
\end{align*}
Integrating the above inequality, we obtain
\begin{align*}
\rho'(\delta)-\rho'(\rho)
&\leq\int_{\rho}^\delta-\frac{\bar R_\lambda}{4(n-1)\lambda}t^{-1}dt\\
&=-\frac{\bar R_\lambda}{4(n-1)\lambda}\ln\frac{\delta}{\rho}.
\end{align*}
Hence we have
\[
\rho'(\rho)
\geq-\lambda+\frac{\bar R_\lambda}{4(n-1)\lambda}\ln\frac{\delta}{\rho}.
\]
A straightforward computation shows that 
\[
\rho\geq\delta\exp\left(
-\frac{4(n-1)\lambda^2}{\bar R_\lambda}\right).
\]
For $\rho<\delta,$ if $\rho'=0$, then by \eqref{3.2}, $\rho''>0$.
Hence, we can apply the estimate on each interval on which $\rho'\leq0$.
Therefore,
\[
\rho(r)
\geq\delta\exp
\left(
-\frac{4(n-1)\lambda^2}{\bar R_\lambda}
\right)>0
\]
throughout the forward interval of existence.
\end{proof}
By definition of $\eta$, \eqref{etap}, and \eqref{etaboundedbelow}, we have
\[
\rho'^2\leq(\eta(0)-\tilde \eta_{\min})\rho^{-(n-2)}.
\]
Hence $(\rho(r),\rho'(r))$ remains in a compact subset of $\{(x,y)\in\mathbb{R}^2~|~x>0\}$.
Therefore, the solution extends to all $r\geq0$.
By \eqref{etap} and \eqref{etaboundedbelow}, there exists $\eta^+$ such that
\[
\eta(\rho(r),\rho'(r))\to\eta^+
\]
as $r\to+\infty$.
Since the trajectory remains in the compact subset throughout its forward interval of existence, for any sequence $t_k\to+\infty$, we can take a subsequence such that
\[
(\rho(t_k),\rho'(t_k))\to(\rho^+,\rho^{++}),\qquad \rho^+>0.
\]
We consider the solution $(u(r),v(r))$ such that 
\[
(u(0),v(0))=(\rho^+,\rho^{++}).
\]
By continuous dependence on initial data, on a small interval $[0,\tau]$,
\[
(\rho(t_k+r),\rho'(t_k+r))\to(u(r),v(r)).
\]
Hence for each $r$, 
\[
\eta(u(r),v(r))=\lim_{k\to+\infty}\eta(\rho(t_k+r),\rho'(t_k+r))
=\eta^+.
\]
Along this solution, $\eta$ is constant.
By \eqref{etap}, $v=0$. By \eqref{3.2} we also have $u=\sqrt{\frac{\bar R}{\lambda}}$.
Hence 
\[
(\rho^+,\rho^{++})=
\left(
\sqrt{\frac{\bar R}{\lambda}},0
\right).
\] 
Since the trajectory remains in a compact subset and has only one accumulation point, we obtain
\[
\rho\to\sqrt{\frac{\bar R}{\lambda}},\quad\rho'\to0
\]
as $r\to+\infty$.

Finally, we show uniqueness.

As we argued before, any nonconstant positive solution on $\mathbb{R}$ satisfies 
\[
\rho\to+\infty,\quad\rho'\to-\lambda
\]
as $r\to-\infty$.

Let $\rho_1$ and $\rho_2$ be two nonconstant positive solutions of $\eqref{3.2}$ defined on $\mathbb{R}$. For sufficiently large $\rho$, let $r_i(\rho)$ denote the inverse of $\rho_i$ on its decreasing interval near $r=-\infty$, and define
\[
p_i(\rho)=\rho'_i(r_i(\rho)),\quad
\rho_i(r_i(\rho))=\rho.
\]
Set $P=p_1-p_2$. A straightforward computation shows
\begin{align}\label{Pp}
\frac{dP}{d\rho}
&=\left\{
-\frac{n-2}{2\rho}
-\frac{\bar R}{2(n-1)\rho p_1p_2}
+\frac{\lambda\rho}{2(n-1)p_1p_2}
\right\}P.
\end{align}
Set 
\[
A(\rho)=-\frac{n-2}{2\rho}
-\frac{\bar R}{2(n-1)\rho p_1p_2}
+\frac{\lambda\rho}{2(n-1)p_1p_2}.
\]
Since 
\[
\frac{A(\rho)}{\rho}\to\frac{1}{2(n-1)\lambda}>0
\]
as $\rho\to\infty$, for sufficiently large $\rho_{++}$ and some positive constant $c>0$,
we obtain
\begin{equation}\label{Acr}
A(\rho)\geq c\rho
\end{equation}
for $\rho\geq \rho_{++}$.
Assume that there exists $\rho_0$ such that $P(\rho_0)\neq0$.
By \eqref{Pp},
\[
P(\rho)=P(\rho_0)\exp\left(\int_{\rho_0}^\rho A(t)dt
\right).
\]
Hence for any $\rho\geq\rho_0$, $P(\rho)\neq0$.
Without loss of generality we may assume that $P>0$ for $\rho>\rho_{++}$.
By \eqref{Pp} and \eqref{Acr}, we obtain
\begin{align*}
\ln P(\rho)-\ln P(\rho_{++})
&=\int_{\rho_{++}}^\rho A(t) dt\\
&\geq c\int_{\rho_{++}}^\rho t dt\\
&=\frac{c}{2}(\rho^2-\rho^2_{++}).
\end{align*}
Therefore, we obtain
\[
P(\rho) \geq P(\rho_{++})\exp\left(\frac{c}{2}(\rho^2-\rho^2_{++})\right).
\]
As $\rho\to+\infty$, the left-hand side of the above inequality tends to zero.
However, the right-hand side of the above inequality tends to infinity, which is a contradiction.
Therefore $p_1\equiv p_2$.

By uniqueness of solutions to the ODE, the solutions coincide up to translation.

Let $(N,\bar g)$ be a complete Riemannian manifold with constant scalar curvature $\bar R$. Then the metric  
$
g=dr^2+\rho^2\bar g,
$
together with the potential function
$\quad F(r)=F(0)+\int_0^r\rho(t)dt$,
defines a complete shrinking gradient Yamabe soliton.
Furthermore,
\[
0<R<\lambda+\sqrt{\frac{\bar R}{(n-1)(n-2)}}.
\]
Furthermore, 
\[
\lim_{r\to-\infty}R=0, \quad \lim_{r\to+\infty} R=\lambda.
\]
\end{proof}

%%%%%%%%%%%%%%%%%%%%%%%%%%
%%%%%%%%%%%%%%%%%%%%%%%%%%
%%%%%%%%%%%%%%%%%%%%%%%%%%
%%%%%%%%%%%%%%%%%%%%%%%%%%
%%%%%%%%%%%%%%%%%%%%%%%%%%

\section{Proof of Theorem \ref{main3}}\label{theorem3}

In this section, we prove Theorem \ref{main3}.

\begin{proof}[Proof of Theorem \ref{main3}]
We only need to consider Case (2) of Theorem \ref{CSZ}. 
 We also choose the orientation of $r$ so that $\rho=F'(r)>0$, and use \eqref{3.2} and \eqref{3.2.2}.
 
Some of the arguments also appear in \cite{Maeta30}, but for completeness, we provide the details in our more general setting.

Fix $\tilde r\in\mathbb{R}$ and set  
\[
\mu(r)=\exp
\left\{
\int_{\tilde r}^r
\left(
\frac{n-1}{\rho(t)}\rho'(t)+\frac{\rho(t)}{2(n-1)}
\right)dt
\right\}.
\]
By \eqref{3.2.2},
\begin{align}\label{2.1}
(\mu\rho'')'=-\mu\frac{\rho'}{n-1}(\lambda+\rho').
\end{align}
Hence, $\mu\rho''$ is decreasing on $\{r\in\mathbb{R}~|~\rho'<0\}$
and increasing on $\{r\in\mathbb{R}~|~0<\rho'<-\lambda\}$.

If $\rho$ is constant, then \eqref{3.2} gives $\bar R=\lambda\rho^2<0$ and $\rho=\sqrt{\frac{\bar R}{\lambda}}$. Hence $g=dr^2+\frac{\bar R}{\lambda}\bar g$. 
In the following, we assume that $\rho$ is nonconstant.  

We show $\rho'>0$ on $\mathbb{R}$.
Assume that there exists $r_0\in\mathbb{R}$ such that $\rho'(r_0)<0$.
If $\rho''(r_0)\leq0$, then \eqref{2.1} implies that $\rho''(r)<0$ and $\rho'(r)<\rho'(r_0)<0$ for all $r>r_0$.
Since $\rho>0$, we have a contradiction.
Therefore, we obtain $\rho''>0$ on $\{r\in\mathbb{R}~|~\rho'(r)<0\}$.
Let $I_1=(r_1,r_2)$ be a maximal interval on which $\rho'<0$.
If $r_1$ is finite, then since $\rho''>0$ on $I_1$, it cannot occur that $\rho'(r_1)=0$. This means $I_1=(-\infty,r_2)$.
Since $\rho'<0$ and $\rho''>0$ on $I_1=(-\infty,r_2)$, we have $\rho(r)\to+\infty$ as $r\to-\infty$.

Set $s=-r$. We also use the same notation $\rho$ for the reparametrized function and $\dot\rho(s)=\frac{d}{ds}\rho(s)$.
In this setting, 
\[
\dot\rho(s)>0,\quad\ddot\rho(s)>0,\quad\text{and}\quad\rho(s)\to+\infty\quad(\text{as}~s\to+\infty).
\]
By \eqref{3.2},
\begin{equation}\label{6.1}
-\dot\rho\rho^2+(n-1)(n-2)\dot\rho^2+2(n-1)\rho\ddot\rho=\bar R-\lambda\rho^2.
\end{equation}
For sufficiently large $s$, we have $\bar R-\lambda\rho^2>0$. Hence, by \eqref{6.1}, 
\begin{align*}
\frac{d}{ds}
\left(
\rho^{\frac{n-2}{2}}\dot\rho
\right)
&=\rho^\frac{n-4}{2}
\left\{
\frac{n-2}{2}\dot\rho^2+\rho\ddot\rho
\right\}\\
&>\frac{1}{2(n-1)}\rho^\frac{n}{2}\dot\rho\\
&=\frac{1}{(n-1)(n+2)}\frac{d}{ds}\rho^\frac{n+2}{2}.
\end{align*}
Hence we have
\[
\dot\rho>\frac{1}{(n-1)(n+2)}\rho^2+C\rho^{-\frac{n-2}{2}}.
\]
For sufficiently large $s_\ell$, on $(s_\ell,+\infty)$, 
\[
\dot\rho
\geq
\frac{1}{2(n-1)(n+2)}\rho^2.
\]
Hence we obtain
\[
-\rho^{-1}(s)+\rho^{-1}(s_\ell)\geq\frac{1}{2(n-1)(n+2)}(s-s_\ell).
\]
Letting $s\to+\infty$, the right-hand side tends to $+\infty$, which is a contradiction.

Therefore, $\rho'\geq0$ on $\mathbb{R}$.
If there exists $r_0\in\mathbb{R}$ such that $\rho'(r_0)=0$, by Lemma \ref{zero}, we have $\rho'\equiv0$ on $\mathbb{R}$.
Hence nonconstant solutions satisfy $\rho'>0$ on $\mathbb{R}$.

\begin{claim}
The scalar curvature $\bar R$ of $N$ is nonpositive.
\end{claim}
\begin{proof}
Assume that $\bar R>0$.
Since $\rho>0$ and $\rho'>0$, set $\alpha=\lim_{r\to-\infty}\rho(r)$.

For any $r_1\in\mathbb{R}$, set $R_1=\rho(r_1)$. For any $r\leq r_1$, we have
$\rho(r)\leq R_1$.

For sufficiently small $\delta>0$ such that
\[
R_1^2\delta+(n-1)(n-2)\delta^2\leq\frac{\bar R}{2},
\]
we consider an interval on which $\rho'\leq\delta$. 
We remark that since 
\[
\int_{-\infty}^{r_1}
\rho'(t)dt=\rho(r_1)-\alpha<+\infty,
\]
for any $\epsilon>0$, it cannot occur that $\rho'(r)>\epsilon$ on $(-\infty,r_2)$ for some $r_2$. 

On $\{r\in\mathbb{R}~|~\rho'(r)\leq\delta,r\leq r_1\}$, by \eqref{3.2}, we have
\begin{align*}
\rho''
&\geq\frac{1}{2(n-1)\rho}
\{
\bar R-\delta R_1^2-(n-1)(n-2)\delta^2
\}\\
&\geq\frac{\bar R}{4(n-1)R_1}>0.
\end{align*}
Choose $r_{-1}<r_1$ such that $0<\rho'(r_{-1})<\delta$.
The above inequality implies that $0<\rho'(r)<\delta$ for all $r\leq r_{-1}$.
Hence we have
\[
\rho'(r_{-1})-\rho'(r)=\int_{r}^{r_{-1}}\rho''(t)dt\geq\frac{\bar R}{4(n-1) R_1}(r_{-1}-r).
\]
Letting $r\to-\infty$, we have a contradiction because $\rho'>0$.
\end{proof}

We obtain $\bar R\leq0$. We now prove the following claim.
\begin{claim}
We have 
\[
0<\rho'<-\lambda,
\quad
\rho''>0.
\]
\end{claim}
\begin{proof}
At $r_1\in\mathbb{R}$ such that $\rho'(r_1)=-\lambda$, by \eqref{3.2}, we have $\rho''(r_1)<0$. 
Hence if there exists $\tilde r\in\mathbb{R}$ such that $\rho'(\tilde r)<-\lambda$, we have $\rho'<-\lambda$ on $(\tilde r,+\infty)$.

If there exists $\tilde{\tilde r}\in\mathbb{R}$ such that $\rho'(\tilde{\tilde r})\geq-\lambda$, then $\rho'(r)\geq-\lambda>0$ for all $r<\tilde{\tilde r}$. However, the positivity of $\rho$  
yields a contradiction.
Therefore, we obtain
\[
0<\rho'<-\lambda.
\]
By \eqref{2.1}, $\mu\rho''$ is increasing. Assume that there exists $r_0\in\mathbb{R}$ such that $\rho''(r_0)\leq 0$. For $r\leq r_0$, we have
\[
\mu(r)\rho''(r)\leq \mu(r_0)\rho''(r_0)\leq0.
\]
Hence $\rho''(r)<0$ for $r<r_0$. The positivity of $\rho$ yields a contradiction.
Therefore, $\rho''>0$ on $\mathbb{R}$.
\end{proof}

Since 
\[
0<\rho'<-\lambda,
\quad
\rho''>0,
\]
we have
\[
\lim_{r\to-\infty}
\rho'(r)=0.
\]
By \eqref{3.2}, 
\[
\rho>\sqrt{\frac{\bar R}{\lambda}}.
\]
Set 
\[
\lim_{r\to-\infty}\rho(r)=\alpha.
\]
We remark that $\alpha\geq\sqrt{\frac{\bar R}{\lambda}}$.
If $\alpha>\sqrt{\frac{\bar R}{\lambda}}$, then by \eqref{3.2}, we have
\[
\rho''\to\frac{\bar R-\lambda\alpha^2}{2(n-1)\alpha}>0
\]
as $r\to-\infty$. The positivity of $\rho'$ yields a contradiction.
Therefore, we obtain
\[
\lim_{r\to-\infty}\rho(r)=\sqrt{\frac{\bar R}{\lambda}},\quad
\lim_{r\to-\infty}\rho'(r)=0.
\]

Since $\rho'>0$ and $\rho''>0$, $\rho(r)\to+\infty$ as $r\to+\infty$.

Since $\rho'$ is increasing and bounded above by $-\lambda$, 
the limit 
\[
\lim_{r\to+\infty}\rho'(r)=\beta\leq-\lambda
\]
exists.
If $\beta<-\lambda$, then by \eqref{3.2}, we have $\rho''(r)\to+\infty$ as $r\to+\infty$. Since $0<\rho'<-\lambda$, we have a contradiction.
Therefore, we obtain 
\[
\lim_{r\to+\infty}\rho'(r)=-\lambda.
\]

From now on, we show the existence and uniqueness of the nonconstant solution. 

We first show the uniqueness of the solutions.
We consider $\bar R\leq0$ and 
\[
\sqrt{\frac{\bar R}{\lambda}}<\rho(r), \quad0<\rho'(r)<-\lambda,\quad \rho''>0
\]
on $\mathbb{R}$. 

Set $G(\rho)=\rho'(r)$, where $\sqrt{\frac{\bar R}{\lambda}}<\rho<+\infty$.
We have $\rho''=G'(\rho)G(\rho)$.
The equation \eqref{3.2} is equivalent to
\begin{align}\label{18.1}
G'(\rho)
=\frac{1}{2(n-1)\rho G(\rho)}
\left\{
\bar R-\lambda\rho^2-G(\rho)\rho^2-(n-1)(n-2)G^2(\rho)
\right\}.
\end{align} 
Let $G_1$ and $G_2$  be two such solutions. We remark that
\[
G_i>0,\quad\lim_{\rho\to\sqrt{\frac{\bar R}{\lambda}}}G_i=0
\]
for $i=1,2$.
Set $P=G_1-G_2$.
A straightforward computation shows 
\[
P'=
-J(\rho)P,
\]
where
\[
J(\rho)
=
\frac{\bar R-\lambda\rho^2}{2(n-1)\rho G_1G_2}+\frac{n-2}{2\rho}>0.
\]
For $\sqrt{\frac{\bar R}{\lambda}}<\rho_1<\rho$,
\[
P(\rho)
=P(\rho_1)\exp
\left\{
-\int_{\rho_1}^\rho J(t)dt
\right\}.
\]
In particular, we have $|P(\rho)|\leq|P(\rho_1)|$. Taking $\rho_1\to\sqrt{\frac{\bar R}{\lambda}}$, we have $P(\rho_1)\to0$.
Hence $P(\rho)=0$ and $G_1=G_2$.

We show the existence of a solution.
Fix $\bar R\leq0$ and $\lambda<0$. 
For any $\varepsilon>0$, we consider the initial conditions
\[
\rho_\varepsilon(0)=\sqrt{\frac{\bar R}{\lambda}}+\varepsilon,\quad
\rho'_\varepsilon(0)=0.
\]
By \eqref{3.2}, we have
$\rho''_\varepsilon>0$ at $0$. Hence on $(0,\tilde \varepsilon)$ for some $\tilde\varepsilon>0$, 
\[
\rho_\varepsilon>\sqrt{\frac{\bar R}{\lambda}}+\varepsilon,\quad
0<\rho'_\varepsilon<-\lambda.
\]
If $\rho_\varepsilon'=0$ and $\rho_\varepsilon>\sqrt{\frac{\bar R}{\lambda}}$ at some point,
then we have $\rho_\varepsilon''>0$ at the point.

If $\rho_\varepsilon'=-\lambda$ at some point, then $\rho_\varepsilon''<0$.
Let $[0,T_\varepsilon)$ be the maximal forward interval of existence. Then,
\[
0<\rho'_{\varepsilon}(r)<-\lambda,
\quad
\sqrt{\frac{\bar R}{\lambda}}+\varepsilon
<
\rho_\varepsilon(r)
<
 \sqrt{\frac{\bar R}{\lambda}}+\varepsilon-\lambda r\qquad(0<r<T_\varepsilon).
\] 
Since $\rho_\varepsilon$ stays bounded away from zero and $(\rho_\varepsilon,\rho'_\varepsilon)$ remains bounded on every finite interval of existence, the standard continuation theorem implies that $\rho_\varepsilon$ extends to all $r\geq0$.

By \eqref{2.1}, $\rho_\varepsilon''>0$ is also preserved.
Therefore, $\rho_\varepsilon\to+\infty$ as $r\to+\infty$.

Set $\rho=\rho_\varepsilon(r)$ and define $G_\varepsilon(\rho)=\rho'_\varepsilon(r)$. 
By the initial condition of $\rho_\varepsilon$, 
\[
G_\varepsilon
\left(
\sqrt{\frac{\bar R}{\lambda}}+\varepsilon
\right)
=\rho'_\varepsilon(0)=0.
\]
We also have
\[
0<G_\varepsilon(\rho)<-\lambda
\]
for $\sqrt{\frac{\bar R}{\lambda}}+\varepsilon<\rho$.

Set $0<\varepsilon_1<\varepsilon_2$. 
By the definition, we have
\[
G_{\varepsilon_1}
\left(
\sqrt{\frac{\bar R}{\lambda}}+\varepsilon_2
\right)>
G_{\varepsilon_2}
\left(
\sqrt{\frac{\bar R}{\lambda}}+\varepsilon_2
\right)=0.
\]
By the uniqueness of the initial value problem, we have
\[
G_{\varepsilon_1}(\rho)>G_{\varepsilon_2}(\rho)\qquad
\left(
\rho>\sqrt{\frac{\bar R}{\lambda}}+\varepsilon_2
\right).
\]
This yields that $G_\varepsilon$ is increasing as $\varepsilon\to0$.
Since $G_\varepsilon<-\lambda$, there exists 
\[
\tilde G(\rho)
=\lim_{\varepsilon\to0}G_\varepsilon(\rho)
\]
for $\rho>\sqrt{\frac{\bar R}{\lambda}}$.

For each $\rho>\sqrt{\frac{\bar R}{\lambda}}$, choose $\varepsilon_0$ such that $0<\varepsilon_0<\rho-\sqrt{\frac{\bar R}{\lambda}}$.
We have $\tilde G(\rho)\geq G_{\varepsilon_0}(\rho)>0$.
Therefore, we obtain
\[
0<\tilde G(\rho)\leq-\lambda.
\]
We show that $\tilde G(\rho)$ also satisfies \eqref{18.1}.
\begin{claim}
$\tilde G(\rho)$ satisfies \eqref{18.1}.
\end{claim}

\begin{proof}
Fix $a,b$ such that 
\[
\sqrt{\frac{\bar R}{\lambda}}<a<b<+\infty.
\]
Choose $0<\tilde\varepsilon_0<a-\sqrt{\frac{\bar R}{\lambda}}$ and set 
\[
\eta=\min_{\rho\in[a,b]}G_{\tilde\varepsilon_0}(\rho)>0.
\]
For $0<\varepsilon<\tilde\varepsilon_0$ and $\rho\in[a,b]$, we have
\[
0<\eta\leq G_{\tilde\varepsilon_0}(\rho)
<G_{\varepsilon}(\rho)
<-\lambda.
\]
Set
\[
\tilde{\tilde G}(\rho,G)
=\frac{1}{2(n-1)\rho G}
\left\{
\bar R-\lambda\rho^2-G\rho^2-(n-1)(n-2)G^2\right\}.
\]
Then we have
\[
G_\varepsilon(\rho)
=G_\varepsilon(a)+\int_{a}^\rho\tilde{\tilde G}(t,G_\varepsilon(t))dt.
\]
Since $\tilde{\tilde G}(t,G_\varepsilon(t))$ is uniformly bounded for $t\in[a,b]$, taking $\varepsilon\to0$, we have
\[
\tilde G(\rho)
=\tilde G(a)+\int_{a}^\rho\tilde{\tilde G}(t,\tilde G(t))dt.
\]
Therefore, $\tilde G$ satisfies \eqref{18.1} and is also smooth.

\end{proof}

Finally, we construct from $\tilde G$ a positive solution $\rho(r)$ of \eqref{3.2} defined on all of $\mathbb{R}$.
A straightforward computation shows that 
\[
(n-1)\frac{d}{d\rho}(\rho^{n-2}G_\varepsilon^2(\rho))
\leq\rho^{n-3}(\bar R-\lambda\rho^2).
\]
By integrating both sides, we have
\begin{align*}
(n-1)\int_{\sqrt{\frac{\bar R}{\lambda}}+\varepsilon}^\rho\frac{d}{dt}(t^{n-2}G_\varepsilon^2(t))dt
&\leq
\int_{\sqrt{\frac{\bar R}{\lambda}}+\varepsilon}^\rho
-\lambda t^{n-3}
\left(
t+\sqrt{\frac{\bar R}{\lambda}}
\right)
\left(
t-\sqrt{\frac{\bar R}{\lambda}}
\right)dt\\
&\leq
\int_{\sqrt{\frac{\bar R}{\lambda}}+\varepsilon}^\rho
-2\lambda t^{n-2}
\left(
t-\sqrt{\frac{\bar R}{\lambda}}
\right)dt\\
&\leq
\int_{\sqrt{\frac{\bar R}{\lambda}}+\varepsilon}^\rho
-2\lambda \rho^{n-2}
\left(
t-\sqrt{\frac{\bar R}{\lambda}}
\right)dt\\
&\leq-\lambda\rho^{n-2}
\left(
\rho-\sqrt{\frac{\bar R}{\lambda}}
\right)^2,
\end{align*}
where the second inequality follows from $\sqrt{\frac{\bar R}{\lambda}}\leq t$.
Therefore, we obtain
\[
G_\varepsilon(\rho)\leq\sqrt{\frac{-\lambda}{n-1}}
\left(
\rho-\sqrt{\frac{\bar R}{\lambda}}
\right).
\]
Letting $\varepsilon\to0$, we obtain
\begin{equation}\label{29.1}
\tilde G(\rho)\leq\sqrt{\frac{-\lambda}{n-1}}
\left(
\rho-\sqrt{\frac{\bar R}{\lambda}}
\right).
\end{equation}
Set $c>\sqrt{\frac{\bar R}{\lambda}}$ and
\[
T(\rho)
=
\int_c^\rho\frac{1}{\tilde G(t)}dt\quad
\left(
\rho>\sqrt{\frac{\bar R}{\lambda}}
\right).
\]
For $\sqrt{\frac{\bar R}{\lambda}}<\rho<c$, \eqref{29.1} gives 
\begin{align*}
T(\rho)
&\leq
-\int_{\rho}^c\frac{1}{\sqrt{\frac{-\lambda}{n-1}}
\left(t-\sqrt{\frac{\bar R}{\lambda}}
\right)}
dt\\
&=
-\sqrt{-\frac{n-1}{\lambda}}
\left\{
\ln
\left(
c-\sqrt{\frac{\bar R}{\lambda}}
\right)
-\ln
\left(
\rho-\sqrt{\frac{\bar R}{\lambda}}
\right)
\right\}.
\end{align*}
Letting $\rho\to\sqrt{\frac{\bar R}{\lambda}}$, the right-hand side tends to $-\infty$.
Hence, $T(\rho)\to-\infty$ as $\rho\to\sqrt{\frac{\bar R}{\lambda}}$.

On the other hand, for $\rho\geq c$, since $\tilde G\leq-\lambda$,
\[
T(\rho)\geq-\frac{1}{\lambda}(\rho-c).
\]
Hence, $T(\rho)\to+\infty$ as $\rho\to+\infty$.

We remark that $T'=\frac{1}{\tilde G(\rho)}>0$.  
Set $\rho(r)=T^{-1}(r)$. Then we have $\rho(r)>\sqrt{\frac{\bar R}{\lambda}}$, $\rho(0)=c$, $\rho'(r)=\tilde G(\rho)$, and $\rho''(r)=\tilde G(\rho)\tilde G'(\rho)$.
Substituting these into \eqref{18.1}, we obtain \eqref{3.2}.  

\end{proof}

%\section*{Data availability statement} 
%Data sharing is not applicable to this article because no datasets were generated or analyzed during the current study.

\section*{Acknowledgments}
The author would like to express deep gratitude to Beomjun Choi for his useful comments and suggestions, and to Eric Chen for helpful comments.

\subsection*{Declaration}
The author used ChatGPT (version 6) to check the manuscript for grammatical and typographical errors.

\end{document}